\documentclass[12pt, leqno, abstracton]{scrartcl}

\usepackage{amssymb}
\usepackage{amsmath}
\usepackage{amscd}
\usepackage{amsthm}

\newtheorem{theorem}{Theorem}
\newtheorem{lemma}{Lemma}

\theoremstyle{definition}
\newtheorem{definition}{Definition}
\newtheorem{remark}{Remark}
\newtheorem{example}{Example}

\newcommand{\Hom}{\operatorname{Hom}}

\newcommand{\End}{\operatorname{End}}
\newcommand{\Spec}{\operatorname{Spec}}
\newcommand{\coh}{\operatorname{coh}}

\newcommand{\Ext}{\operatorname{Ext}}
\newcommand{\map}{\rightarrow}
\newcommand{\Map}{\longrightarrow}

\newcommand{\ol}{\overline}
\newcommand{\Gr}{\operatorname{Gr}}
\newcommand{\gr}{\operatorname{gr}}
\newcommand{\qgr}{\operatorname{qgr}}
\newcommand{\QGr}{\operatorname{QGr}}
\newcommand{\ul}{\underline}
\newcommand{\id}{\operatorname{id}}

\newcommand{\Proj}{\operatorname{Proj}}
\DeclareMathOperator{\adeg}{\widehat{deg}}
\newcommand{\Tors}{\operatorname{Tors}}
\newcommand{\cd}{\operatorname{cd}}

\begin{document}

\title{Arakelov theory of noncommutative arithmetic surfaces}

\author{Thomas Borek}

\maketitle

\begin{abstract}
The purpose of the present article is to initiate Arakelov theory of noncommutative arithmetic surfaces. Roughly speaking, a noncommutative arithmetic surface is a noncommutative projective scheme of cohomological dimension 1 of finite type over $\Spec\mathbb Z.$ An important example is the category of coherent right $\mathcal O$-modules, where $\mathcal O$ is a coherent sheaf of $\mathcal O_X$-algebras and $\mathcal O_X$ is the structure sheaf of a commutative arithmetic surface $X.$ Since smooth hermitian metrics are not available in our noncommutative setting, we have to adapt the definition of arithmetic vector bundles on noncommutative arithmetic surfaces. Namely, we consider pairs $(\mathcal E,\beta)$ consisting of a coherent sheaf $\mathcal E$ and an automorphism $\beta$ of the real sheaf $\mathcal E_\mathbb R$ induced by $\mathcal E.$ We define the intersection of two such objects using the determinant of the cohomology and prove a Riemann-Roch theorem for arithmetic line bundles on noncommutative arithmetic surfaces.
\end{abstract}



\newpage

\section{Introduction}
In recent years, number theory and arithmetic geometry have been enriched by new techniques from noncommutative geometry. For instance, Consani and Marcolli show how noncommutative (differential) geometry \`a la Connes provides a unified description of the archimedean and the totally split degenerate fibres of an arithmetic surface. We refer the interested reader to \cite{CM1}, \cite{CM2}, \cite{CM3}, and \cite{Ma}. Until now, only ideas and methods of noncommutative geometry in the form developed by Connes \cite{C} have been applied to number theory and arithmetic geometry. But as Marcolli mentions in the last chapter of her book \cite{Ma}, in which she addresses the question ``where do we go from here?'', it would be interesting to consider more algebraic versions of noncommutative geometry in the context of number theory and arithmetic geometry. Our work is intended as a first step in this direction.

In a preceding paper \cite{B}, we have established a version of Arakelov theory of noncommutative arithmetic curves. The present article continues our work on noncommutative arithmetic geometry by introducing a version of Arakelov theory of noncommutative arithmetic surfaces. To define noncommutative arithmetic surfaces, we use the well developed theory of noncommutative projective schemes. The standard reference for noncommutative projective schemes is the article \cite{AZ} of Artin and Zhang. Noncommutative algebraic geometry was mainly developed by Artin, Tate and van den Bergh \cite{ATV}, by Manin \cite{M}, and by Kontsevich and Rosenberg \cite{Ro}. The general philosophy of noncommutative algebraic geometry is that noncommutative spaces are made manifest by the modules that live on them in the same way that the properties of a commutative scheme $X$ are manifested by the category of (quasi)-coherent $\mathcal O_X$-modules. The modules over a noncommutative space form, by definition, an abelian category. The category is the basic object of study in noncommutative algebraic geometry. In short, a noncommutative space is an abelian category.

Noncommutative \emph{projective} schemes are determined by specific abelian categories. Let $A$ be an $\mathbb N$-graded right noetherian algebra over a commutative noetherian ring $k,$ let $\Gr A$ be the category of graded right $A$-modules and $\Tors A$ its full subcategory generated by the right bounded modules. The noncommutative projective scheme associated to $A$ is the pair $\Proj A=(\QGr A,\mathcal A),$ where $\QGr A$ is the quotient category $\Gr A/\Tors A$ and $\mathcal A$ the image of $A$ in $\QGr A.$ This definition is justified by a theorem of Serre \cite{Se}, which asserts that if $A$ is a commutative graded algebra generated in degree 1 and $X=\Proj A$ the associated projective scheme, then the category of quasi-coherent $\mathcal O_X$-modules is equivalent to the quotient category $\QGr A.$ Thus the objects in $\QGr A$ are the noncommutative geometric objects analogous to sheaves of $\mathcal O_X$-modules.

Working with the Grothendieck category $\QGr A$ allows to define a lot of useful tools known from the theory of commutative projective schemes, such as ample sheaves, twisting sheaf, cohomology, dualizing object, Serre duality and so on. However, equating noncommutative schemes with abelian categories has drawbacks: most notably there are no points, so the powerful tool of localization is not available. This forces to search for equivalent global definitions of objects which are usually defined locally. Secondly, at present there is almost no connection between noncommutative algebraic geometry considered in the just described way and noncommutative geometry developed by Connes. In particular it is not clear what should be the (noncommutative) differential structure on a noncommutative algebraic variety over $\mathbb C.$ The lack of a differential structure makes it impossible to define \emph{smooth} hermitian metrics on locally free sheaves, which is a serious problem because hermitian vector bundles are essential in Arakelov theory. Fortunately, we found a substitute for the metrics. Namely, we consider pairs $(\mathcal E,\beta)$ consisting of a coherent sheaf $\mathcal E$ and an automorphism $\beta$ of the real sheaf $\mathcal E_\mathbb R$ induced by $\mathcal E.$ It turns out that we can formulate a concise theory using such pairs instead of hermitian vector bundles.

In Section 2 we review the definition and construction of noncommutative projective schemes and provide some results which are needed in the following sections. At the beginning of the third section, we introduce noncommutative arithmetic surfaces and provide some examples. Thereafter we give the precise definition of arithmetic vector bundles on noncommutative arithmetic surfaces. Section 4 is concerned with intersection theory on noncommutative arithmetic surfaces. We define the intersection of two arithmetic line bundles using the determinant of the cohomology. Moreover, we show that the determinant of the cohomology is compatible with Serre duality and we apply this fact to prove a Riemann-Roch theorem for arithmetic line bundles on noncommutative arithmetic surfaces.

\section{Preliminaries}
Firstly, we introduce the terminology and notation from graded ring and module theory which we will use throughout. In this paper $k$ will denote a noetherian commutative ring. A $\mathbb Z$-graded $k$-module $M=\bigoplus_{i\in\mathbb Z}$ is called \emph{locally finite} if each component $M_i$ is a finitely generated $k$-module. Let $A=\bigoplus_{i\in\mathbb Z}A_i$ be a $\mathbb Z$-graded $k$-algebra. If each component $A_i$ is a finitely generated $k$-module, then $A$ is called aThe category of $\mathbb Z$-graded right $A$-modules is denoted by $\Gr A.$ In this category homomorphisms are of degree zero; thus if $M=\bigoplus_{i\in\mathbb Z}M_i$ and $N=\bigoplus_{i\in\mathbb Z}N_i$ are $\mathbb Z$-graded right $A$-modules and $f\in\Hom_{\Gr A}(M,N),$ then $f(M_i)\subset N_i$ for all integers $i.$ Given a $\mathbb Z$-graded $A$-module $M$ and an integer $d,$ the $A_0$-module $\bigoplus_{i\ge d}M_i$ is denoted by $M_{\ge d}.$ If $A$ is $\mathbb N$-graded, then $M_{\ge d}$ is a $\mathbb Z$-graded $A$-submodule of $M,$ and $A_{\ge d}$ is an $\mathbb N$-graded two-sided ideal of $A.$ In the following, ``graded'' without any prefix will mean $\mathbb Z$-graded.

For a graded $A$-module $M=\bigoplus_{i\in\mathbb Z}M_i$ and any integer $d,$ we let $M[d]$ be the graded $A$-module defined by $M[d]_i=M_{i+d}$ for all $i\in\mathbb Z.$ It is clear that the rule $M\mapsto M[d]$ extends to an automorphism of $\Gr A.$ We call $[1]$ the \emph{degree-shift functor.} Note $[d]=[1]^d.$

Let $A$ be a right noetherian $\mathbb N$-graded $k$-algebra. We say that an element $x$ of a graded right $A$-module $M$ is \emph{torsion} if $xA_{\ge d}=0$ for some $d.$ The torsion elements in $M$ form a graded $A$-submodule which we denote by $\tau(M)$ and call the \emph{torsion submodule} of $M.$ A module $M$ is called \emph{torsion-free} if $\tau(M)=0$ and \emph{torsion} if $M=\tau(M).$ For all short exact sequences $0\map M' \map M \map M'' \map 0$ in $\Gr A,$ $M$ is torsion if and only if $M'$ and $M''$ are, so the torsion modules form a dense subcategory of $\Gr A$ for which we will use the notation $\Tors A.$ We set $\QGr A$ for the quotient category $\Gr A/\Tors A;$ the formal definition of this category can be found in \cite{G} or \cite{Po}, but roughly speaking, $\QGr A$ is an abelian category having the same objects as $\Gr A$ but objects in $\Tors A$ become isomorphic to $0.$

As for every quotient category there is a quotient functor $\pi$ from $\Gr A$ to $\QGr A.$ Since the category $\Gr A$ has enough injectives and $\Tors A$ is closed under direct sums, there is a section functor $\sigma$ from $\QGr A$ to $\Gr A$ which is right adjoint to $\pi,$ cf. \cite{Po}, Section 4.4. Hence $\Tors A$ is a localizing subcategory of $\Gr A.$ The functor $\pi$ is exact and the functor $\sigma$ is left exact.

We will modify the notation introduced above by using lower case to indicate that we are working with finitely generated $A$-modules. Thus $\gr A$ denotes the category of finitely generated graded right $A$-modules, $\operatorname{tors}A$ denote the dense subcategory of $\gr A$ of torsion modules, and $\qgr A$ the quotient category $\gr A/\operatorname{tors}A.$ The latter is the full subcategory of noetherian objects of $\QGr A,$ cf. \cite{AZ}, Proposition 2.3.

We proceed with a review of the construction of noncommutative projective schemes, and we fix the notation we will use throughout. The standard reference for the theory of noncommutative projective schemes is the article \cite{AZ} of Artin and Zhang. In his fundamental paper \cite{Se}, Serre proved a theorem which describes the coherent sheaves on a projective scheme in terms of graded modules as follows. Let $A$ be a finitely generated commutative $\mathbb N$-graded $k$-algebra and $X=\Proj A$ the associated projective scheme, let $\coh X$ denote the category of coherent sheaves on $X,$ and let $\mathcal O_X(n)$ denote the $n$th power of the twisting sheaf of $X.$ Define a functor $\Gamma_*:\coh X\map\Gr A$ by
$$\Gamma_*(\mathcal F)=\bigoplus_{n\in\mathbb Z}H^0(X,\mathcal F\otimes\mathcal O_X(n))$$
and let $\pi:\Gr A\map\QGr A$ be the quotient functor. Then Serre proved

\begin{theorem}[\cite{Se}, Proposition 7.8]\label{Serre}
Suppose that $A$ is generated over $k$ by elements of degree 1. Then $\pi\circ\Gamma_*$ defines an equivalence of categories $\coh X\map\qgr A.$
\end{theorem}

Since the categories $\QGr A$ and $\qgr A$ are also available when $A$ is not commutative, this observation provides a way to make a definition in a more general setting. Let $A$ be a right noetherian $\mathbb N$-graded $k$-algebra, let $\QGr A$ be the quotient category introduced above and let $\pi:\Gr A\map\QGr A$ be the quotient functor. The associated \emph{noncommutative projective scheme} is the pair $\Proj A=(\QGr A,\mathcal A)$ where $\mathcal A=\pi(A).$ The subcategory $\Tors A$ is stable under the degree-shift functor because $M$ is a torsion module if and only if $M[1]$ is. Hence, by the universal property of quotient categories (see \cite{G}, Corollaire III.2) there is an induced automorphism $s$ of $\QGr A$ defined by the equality $s\circ\pi=\pi\circ[1].$ We call $s$ the \emph{twisting functor} of $\Proj A.$ Given an object $\mathcal M$ of $\QGr A,$ we will often write $\mathcal M[i]$ instead of $s^{i}(\mathcal M).$

There is a representing functor $\Gamma:\QGr A\map\Gr A$ constructed as follows. For an object $\mathcal M$ in $\QGr A,$ we define
$$\Gamma(\mathcal M)=\bigoplus_{i\in\mathbb Z}\Hom_{\QGr A}\left(\mathcal A,\mathcal M[i]\right).$$
If $a\in\Gamma(\mathcal A)_i=\Hom_{\QGr A}(\mathcal A,\mathcal A[i]),$ $b\in\Gamma(\mathcal A)_j=\Hom_{\QGr A}(\mathcal A,\mathcal A[j])$ and $m\in\Gamma(\mathcal M)_d=\Hom_{\QGr A}(\mathcal A,\mathcal M[d]),$ we define multiplications by
$$ab=s^j(a)\circ b \quad\text{and}\quad ma=s^i(m)\circ a.$$
With this law of composition, $\Gamma(\mathcal A)$ becomes a graded $k$-algebra and $\Gamma(\mathcal M)$ a graded right $\Gamma(\mathcal A)$-module. Moreover there is a homomorphism $\varphi:A\map\Gamma(\mathcal A)$ of graded $k$-algebras sending $a\in A_i$ to $\pi(\lambda_a)\in\Gamma(\mathcal A)_i,$ where $\lambda_a:A\map A$ is left multiplication by $a.$ So each $\Gamma(\mathcal M)$ has a natural graded right $A$-module structure, and it is clear that $\Gamma$ defines a functor from $\QGr A$ to $\Gr A.$ The following lemma summarizes some important properties of the representing functor $\Gamma.$

\begin{lemma}\label{lemma1}
Let $\Proj A$ be a noncommutative projective scheme. Then the following statements hold:
\begin{itemize}
\item[(i)] The representing functor $\Gamma:\QGr A\map\Gr A$ is isomorphic to the section functor $\sigma:\QGr A\map\Gr A;$
\item[(ii)] $\Gamma$ is fully faithful;
\item[(iii)] $\pi\Gamma\cong\id_{\QGr A}.$
\end{itemize}
\end{lemma}

\begin{proof}
(i) Given an object $\mathcal M$ of $\QGr A,$ we let $F$ denote the contravariant left exact functor $\Hom_{\QGr A}\left(\pi(-),\mathcal M\right)$ from $\Gr A$ to the category of $k$-modules. It follows from Watt's Theorem (see \cite{YZ}, Theorem 1.3) that $F\cong\Hom_{\Gr A}\left(-,\ul{F}(A)\right),$ where
\begin{equation*}
\ul{F}(A)=\bigoplus_{i\in\mathbb Z}\Hom_{\QGr A}\left(\pi\left(A[-i]\right),\mathcal M\right)\cong\bigoplus_{i\in\mathbb Z}\Hom_{\QGr A}\left(\pi\left(A\right),\mathcal M[i]\right)=\Gamma(\mathcal M).
\end{equation*}
Note that the original version of \cite{YZ}, Theorem 1.3 holds for functors from $\gr A$ to the category of $k$-modules. But in our situation where the functor $F$ converts direct sums into direct products (quotient functors preserve direct sums), the restriction to noetherian objects can be omitted. Hence we obtain a natural isomorphism
$$\Hom_{\QGr A}\left(\pi(-),\mathcal M\right)\cong\Hom_{\Gr A}\left(-,\Gamma(\mathcal M)\right).$$
This is true for every object $\mathcal M$ of $\QGr A,$ thus $\Gamma$ is right adjoint to the quotient functor $\pi,$ whence $\Gamma\cong\sigma$ as claimed.

The statements (ii) and (iii) follow from the respective properties of the section functor $\sigma,$ c.f. \cite{G}, Lemme III.2.1, Lemme III.2.2, and Proposition III.2.3.a.
\end{proof}

Being defined in terms of natural isomorphisms between Ext groups, also for noncommutative projective schemes there is a well-defined notion of Serre duality.  Here we only want to give the relevant definitions; for more details we refer to \cite{YZ}, \cite{J}, and \cite{RV}.

The \emph{cohomological dimension} of a noncommutative projective scheme $\Proj A$ is defined to be
$$\cd(\Proj A)=\max\{i\mid\Ext_{\QGr A}^i(\mathcal A,\mathcal M)\neq 0\text{ for some }\mathcal M\in\QGr A\}.$$
Following Yekutieli and Zhang \cite{YZ}, we say that a noncommutative projective scheme $\Proj A$ of cohomological dimension $d<\infty$ \emph{has a dualizing object,} if there exists an object $\omega$ in $\qgr A$ and a natural isomorphism
\begin{equation}\label{omega}
\theta:\Ext_{\qgr A}^d(\mathcal A,-)^\vee\Map\Hom_{\qgr A}(-,\omega).
\end{equation}
Here the dual is taken in the category of $k$-modules. Clearly, the dualizing object is unique, up to isomorphism, if it exists. Furthermore, if $\Proj A$ has a dualizing object $\omega,$ then for each $0\le i\le d=\cd(\Proj A),$ there is a natural transformation
$$\theta^{i}:\Ext_{\qgr A}^i(-,\omega)\Map\Ext_{\qgr A}^{d-i}(\mathcal A,-)^\vee,$$
where $\theta^0$ is the inverse of the natural isomorphism in (\ref{omega}). We say that $\Proj A$ satisfies \emph{Serre duality} if $\theta^i$ are  isomorphisms for all $i.$

Next, we are going to introduce invertible objects. Let $X$ be a projective scheme. Recall that a coherent $\mathcal O_X$-module $\mathcal L$ is called invertible if $\mathcal L\otimes\mathcal L^\vee\cong\mathcal O_X.$ This means that the functor $\mathcal L\otimes -$ is an autoequivalence of the category of coherent sheaves on $X.$ Moreover
$$\mathcal L[1]=\mathcal L\otimes\mathcal O_X[1]\cong\mathcal O_X[1]\otimes\mathcal L.$$
Hence the autoequivalence $\mathcal L\otimes -$ commutes, up to natural isomorphism, with the twisting functor $\mathcal O_X[1]\otimes -.$ This motivates

\begin{definition}
Let $\Proj A$ be a noncommutative projective scheme. An object $\mathcal L$ of $\qgr A$ is called \emph{invertible} if there exists an autoequivalence $t$ of $\qgr A$ such that $t(\mathcal A)\cong\mathcal L$ and which commutes with the twisting functor $s,$ i.e. $s\circ t\cong t\circ s.$
\end{definition}

Invertible objects are characterized in

\begin{theorem}\label{invertible}
Let $\Proj A$ be a noncommutative projective scheme and let $\mathcal L$ be an object of $\qgr A.$ Then, $\mathcal L$ is invertible if and only if $\Gamma(\mathcal L)$ is a graded $A$-bimodule and $t'=\pi\left(\Gamma(-)\otimes_A\Gamma(\mathcal L)\right)$ is an autoequivalence of $\qgr A$ such that $t'(\mathcal A)\cong\mathcal L.$
\end{theorem}

\begin{proof}
Suppose that $L=\Gamma(\mathcal L)$ is a graded $A$-bimodule and that $t'=\pi\left(\Gamma(-)\otimes_AL\right)$ is an autoequivalence of $\qgr A$ such that $t'(\mathcal A)\cong\mathcal L.$ In order to prove that $\mathcal L$ is invertible, it remains to show that $t'$ commutes with the twisting functor $s.$ We have
$$\Gamma(s(-))\otimes_AL=\Gamma(-)[1]\otimes_AL\cong\left(\Gamma(-)\otimes_AL\right)[1],$$
where the last isomorphism is a basic property of the graded tensor product. This leads to
$$t'\circ s=\pi\left(\Gamma\left(s(-)\right)\otimes_AL\right)\cong\pi\left(\left(\Gamma(-)\otimes_AL\right)[1]\right)=\left(\pi\left(\Gamma(-)\otimes_AL\right)\right)[1]=s\circ t'.$$

To prove the other implication, we suppose that $\mathcal L$ is invertible and that $t$ is an autoequivalence of $\qgr A$ such that $t(\mathcal A)\cong\mathcal L.$ Firstly, we show that $L=\Gamma(\mathcal L)$ is a graded $A$-bimodule. As we do not want to overload the notation, we assume that $t\circ s=s\circ t;$ if the natural isomorphism $\epsilon:t\circ s\map s\circ t$ is not the identity, the forthcoming arguments work as well, unless one always has to take $\epsilon$ into account.

If $a\in\Gamma(\mathcal A)_i=\Hom_{\QGr A}(\mathcal A,\mathcal A[i]),$ $b\in\Gamma(\mathcal A)_j=\Hom_{\QGr A}(\mathcal A,\mathcal A[j])$ and $x\in L_d=\Hom_{\QGr A}(\mathcal A,t(\mathcal A)[d]),$ we define multiplications by
$$ax=t(s^d(a))\circ x\text{ and }xb=s^j(x)\circ b.$$
Since $t\circ s=s\circ t,$ these operations induce a graded $\Gamma(\mathcal A)$-bimodule - and particularly also a graded $A$-bimodule - structure on $L.$

Secondly we are going to prove that the functors $t$ and $t'=\pi\left(\Gamma(-)\otimes_AL\right)$ from $\gr A$ to $\qgr A$ are isomorphic. Let $\mathcal E$ be any object of $\QGr A.$ Since both functors
$$F=\Hom_{\QGr A}\left(t\circ\pi(-),\mathcal E\right)\quad\text{and}\quad G=\Hom_{\QGr A}\left(\pi\left(-\otimes_AL\right),\mathcal E\right)$$
are contravariant left exact functors from $\gr A$ to Mod-$k,$ we may apply Watts' Theorem for $\gr A$ (\cite{YZ}, Theorem 1.3) to conclude
$$F\cong\Hom_{\Gr A}(-,\underline{F}(A))\text{ and }G\cong\Hom_{\Gr A}(-,\underline{G}(A)).$$
Note
$$\underline{F}(A)=\bigoplus_{i\in\mathbb Z}\Hom_{\QGr A}\left(t\circ\pi(A[-i]),\mathcal E\right)\cong\bigoplus_{i\in\mathbb Z}\Hom_{\QGr A}\left(t(\mathcal A)[-i],\mathcal E\right)$$
and
$$\underline{G}(A)=\bigoplus_{i\in\mathbb Z}\Hom_{\QGr A}\left(\pi\left(A[-i]\otimes_AL\right),\mathcal E\right)\cong\bigoplus_{i\in\mathbb Z}\Hom_{\QGr A}\left(\pi(L)[-i],\mathcal E\right).$$
Recall $L=\Gamma(\mathcal L)\cong\Gamma\left(t(\mathcal A)\right)$ and $\pi\circ\Gamma\cong\id_{\qgr A}$ (Lemma \ref{lemma1}(iii)), which shows $\pi(L)\cong t(\mathcal A).$ So, $\ul{F}(A)\cong\ul{G}(A)$ and thus $F\cong G.$ As this holds for every object $\mathcal E$ of $\QGr A,$ the functors  $t\circ\pi$ and $\pi(-\otimes_AL)$ are isomorphic. Therefore $t\cong t\pi\Gamma\cong\pi\left(\Gamma(-)\otimes_AL\right)=t',$ which completes the proof of the theorem.
\end{proof}

At the end of this section we are concerned with base change. Let $K$ be a $k$-algebra, let $A$ be an $\mathbb N$-graded $k$-algebra, and let $A_K$ denote the $\mathbb N$-graded $k$-algebra $A\otimes_kK.$ If $K$ is a flat $k$-module, then the functor $-\otimes_kK:\Gr A\map\Gr A_K$ is exact. The quotient functor $\pi_K:\Gr A_K\map\QGr A_K$ is also exact, so is their composition. Clearly, if a graded right $A$-module $M$ is torsion, the graded right $A_K$-module $M_K=M\otimes_kK$ is torsion as well. Therefore $\pi_K(M_K)=0$ whenever $M$ is torsion. Hence the universal property of quotient categories, see \cite{G}, Corollaire III.2, ensures the existence of a unique functor $F:\QGr A\map\QGr A_K$ such that $\pi_K\circ(-\otimes_kK)=F\circ\pi,$ where, as usual, $\pi$ denotes the quotient functor from $\Gr A$ to $\QGr A.$ Given an object $\mathcal M$ of $\QGr A,$ we will usually write $\mathcal M_K$ instead of $F(\mathcal M).$ With this notation, the following base change lemma holds:

\begin{lemma}\label{basechange}
Suppose that $A$ and $A_K$ are right noetherian and that $K$ is a flat $k$-module. Let $\mathcal M$ and $\mathcal N$ be two objects of $\QGr A.$ If $\mathcal N$ is noetherian, then for all $i\ge 0,$
$$\Ext_{\QGr A_K}^i(\mathcal N_K,\mathcal M_K)\cong\Ext_{\QGr A}^i(\mathcal N,\mathcal M)\otimes_kK.$$
\end{lemma}

\begin{proof}
Let $M$ and $N$ be graded right $A$-modules such that $\mathcal M=\pi(M)$ and $\mathcal N=\pi(N)$ respectively. Since $\qgr A$ is a quotient category of the noetherian objects of $\Gr A,$ we may assume that $N$ is noetherian. Then \cite{AZ}, Proposition 7.2(1), asserts that for all $i\ge 0,$
\begin{equation}\label{bc1}
\underline{\Ext}_{\QGr A}^i(\mathcal N,\mathcal M)\cong\varinjlim_{n}\underline{\Ext}_{\Gr A}^i(N_{\ge n},M).
\end{equation}
Recall $\mathcal N_K=\pi_K(N_K)$ and $\mathcal M_K=\pi_K(M_K).$ Since $N$ is noetherian, the graded right $A_K$-module $N_K$ is finitely generated too, and $A_K$ is right noetherian by assumption, so we may apply \cite{AZ}, Proposition 7.2(1), once again to conclude that for all $i\ge 0,$
\begin{equation}\label{bc2}
\underline{\Ext}_{\QGr A_K}^i(\mathcal N_K,\mathcal M_K)\cong\underset{n}{\underrightarrow{\lim}}\underline{\Ext}_{\Gr A_K}^i\left((N_K)_{\ge n},M_K\right)=\underset{n}{\underrightarrow{\lim}}\underline{\Ext}_{\Gr A_K}^i\left((N_{\ge n})_K,M_K\right).
\end{equation}
But \cite{NV}, Proposition I.2.12, tells us that $\underline{\Ext}_{\Gr A_K}^i\left((N_{\ge n})_K,M_K\right)=\Ext_{A_K}^i\left((N_{\ge n})_K,M_K\right).$ Since $K$ is flat over $k,$ it follows from a change of ring theorem, cf. \cite{Re}, Theorem (2.39), that for all $i$ and $n,$
\begin{equation}\label{bc3}
\Ext_{A_K}^i\left((N_{\ge n})_K,M_K\right)\cong\Ext_{A}^i(N_{\ge n},M)\otimes_kK.
\end{equation}
Combining (\ref{bc1}), (\ref{bc2}), (\ref{bc3}) and the fact that the tensor product commutes with direct limits, leads to
$$\underline{\Ext}_{\QGr A_K}^i(\mathcal N_K,\mathcal M_K)\cong\underline{\Ext}_{\QGr A}^i(\mathcal N,\mathcal M)\otimes_kK.$$
Taking degree zero parts yields the claim.
\end{proof}

\section{Noncommutative arithmetic surfaces and arithmetic vector bundles}
\subsection{Noncommutative arithmetic surfaces: definition and examples}
Following Soul\'e \cite{So}, an arithmetic surface is a regular scheme $X,$ projective and flat over $\Spec\mathbb Z$ of Krull dimension two. We now isolate the conditions that we want to be satisfied by our noncommutative analogues of the homogeneous coordinate ring $S$ of $X$ and of $\coh(X),$ the category of coherent sheaves on $X.$ They clearly satisfy the following conditions:
\begin{itemize}
\item $S$ is noetherian and locally finite, and the associated graded real algebra $S_\mathbb R$ is noetherian too.
\item $\coh(X)$ is $H$-finite. In other words, for every coherent $\mathcal O_X$-modules $\mathcal F$ and all integers $i\ge 0,$ the cohomology groups $H^i(X,\mathcal F)$ are finitely generated;
\item $\coh(X)$ has cohomological dimension 1 in the sense that $H^i(X,\mathcal F)=0$ for every coherent $\mathcal O_X$-module $\mathcal F$ and all $i>1;$
\item $\coh(X)$ has a dualizing sheaf $\omega_X.$
\end{itemize}
This motivates

\begin{definition}
A \emph{noncommutative arithmetic surface} is a noncommutative projective scheme $\Proj A,$ which satisfies the following conditions:
\begin{itemize}
\item $A$ is an $\mathbb N$-graded right noetherian locally finite $\mathbb Z$-algebra such that the associated real algebra $A_\mathbb R=A\otimes_\mathbb Z\mathbb R$ is also right noetherian;
\item $\qgr A$ is H-finite and has cohomological dimension 1, i.e. for every object $\mathcal M$ in $\qgr A$ and all $i\ge 0,$ the Ext groups $\Ext_{\qgr A}^i(\mathcal A,\mathcal M)$ are finitely generated and $\Ext_{\qgr A}^i(\mathcal A,\mathcal M)=0$ whenever $i>1;$
\item $\qgr A$ has a dualizing object $\omega.$
\end{itemize}
\end{definition}

We are going to provide some examples of noncommutative arithmetic surfaces.

\begin{example}[Noncommutative arithmetic surfaces derived from commutative ones]
Let $X$ be a commutative arithmetic surface, let $\mathcal O$ be a coherent sheaf of $\mathcal O_X$-algebras, and let $\coh(\mathcal O)$ denote the category of coherent sheaves with a structure of right $\mathcal O$-module. We claim that $\coh(\mathcal O)$ is a noncommutative arithmetic surface. To verify this, we have to check the several conditions that are imposed in the definition of a noncommutative arithmetic surface.

Firstly, the homogeneous coordinate ring of $\coh(\mathcal O)$ is defined to be
$$B=\bigoplus_{n\in\mathbb N}H^0(X,\mathcal O[n]),$$
where $\mathcal O[n]=\mathcal O\otimes_{\mathcal O_X}\mathcal O_X[n].$ For every object $\mathcal M$ of $\coh(\mathcal O),$ we have $\Hom_{\mathcal O}(\mathcal O,\mathcal M)=H^0(X,\mathcal M),$ therefore the fact that $X$ $-$ or more precisely the triple $(\coh(\mathcal O_X),\mathcal O_X,$ \linebreak $-\otimes_{\mathcal O_X}\mathcal O_X[n])$ $-$ satisfies the conditions (H1), (H2)' and (H3) of Corollary 4.6 of \cite{AZ} implies that the triple  $\left(\coh(\mathcal O),\mathcal O,-\otimes_\mathcal O\mathcal O[n]\right)$ fulfils them as well. Hence \cite{AZ}, Corollary 4.6, tells us that $B$ is a right noetherian locally finite $\mathbb Z$-algebra and the pair $(\coh(\mathcal O),\mathcal O)$ is isomorphic to the noncommutative projective scheme $\Proj B.$ In particular, the categories $\coh(\mathcal O)$ and $\qgr B$ are equivalent. Since the homogeneous coordinate ring $S$ of $X$ is noetherian and $\mathcal O$ is a coherent $\mathcal O_X$-module, it follows that $B$ is a finitely generated $\mathbb N$-graded $S$-module. After tensoring with $\mathbb R,$ $S_\mathbb R$ is still noetherian and $B_\mathbb R$ finitely generated over $S_\mathbb R,$ so also
 $B_\mathbb R$ is right noetherian.

Secondly, let $M$ be a finitely generated graded right $B$-module and denote by $\mathcal M$ and $\mathcal M_\mathcal S$ the corresponding objects of $\qgr B$ and $\qgr S.$ Combining Proposition 3.11(3) and Theorem 8.3(2),(3) of \cite{AZ} shows that for all $i\ge 0,$
\begin{equation}\label{cohom}
\Ext_{\qgr B}^i(\mathcal B,\mathcal M)\cong\Ext_{\qgr S}^i(\mathcal S,\mathcal M_\mathcal S).
\end{equation}
On the other hand it follows from Serre's Theorem, see Theorem \ref{Serre}, that the categories $\qgr S$ and $\coh(X)$ are equivalent, therefore $\qgr S$ is $H$-finite and has cohomological dimension one. The isomorphism in (\ref{cohom}) shows that $\Proj B$ inherits these properties from $\qgr S.$

It remains to prove that $\coh(\mathcal O)$ has a dualizing object. If $\omega_X$ denotes the dualizing sheaf of $X,$ then there are natural isomorphism
$$\Ext_\mathcal O^1(\mathcal O,-)^\vee=H^1(X,-)^\vee\cong\Hom_{X}(-,\omega_X)$$
of functors from $\coh(\mathcal O)$ to the category of abelian groups. This implies that the $\mathcal Hom$ sheaf $\mathcal Hom_{\mathcal O_X}(\mathcal O,\omega_X)$ is a representing object of the functor $\Ext_\mathcal O^1(\mathcal O,-)^\vee$ and therefore a dualizing object of $\coh(\mathcal O).$

All together this shows that $\coh(\mathcal O)$ is indeed a noncommutative arithmetic surface. The noncommutative projective scheme $\coh(\mathcal O)$ is an example of what Artin and Zhang \cite{AZ} call classical projective schemes.
\end{example}

\begin{example}[Maximal orders]
We now specialize the last example a little bit. Let $K$ be the function field of a commutative arithmetic surface $X$ and let $\mathcal O$ be a sheaf of maximal $\mathcal O_X$-orders in a finite dimensional semisimple $K$-algebra $A.$ Then, the dualizing sheaf $\omega=\mathcal Hom_{\mathcal O_X}(\mathcal O,\omega_X)$ is an invertible $\mathcal O$-module. Since $\mathcal Hom_{\mathcal O_X}(\mathcal O,\omega_X)\cong\mathcal Hom_{\mathcal O_X}(\mathcal O,\mathcal O_X)\otimes_{\mathcal O_X}\omega_X$ and $\omega_X$ is an invertible sheaf, it suffices to show that $\widetilde{\mathcal O}=\mathcal Hom_{\mathcal O_X}(\mathcal O,\mathcal O_X)$ is an invertible $\mathcal O$-module. We prove this locally. So let $x\in X.$ Then $\widetilde{\mathcal O}_x\cong\Hom_{\mathcal O_{X,x}}(\mathcal O_x,\mathcal O_{X,x}).$ Since $\mathcal O_x$ is an $\mathcal O_{X,x}$-order in the semisimple $K$-algebra $A,$ it follows that $\widetilde{\mathcal O}_x$ is isomorphic to the inverse different $\left\{a\in A\mid tr_{A\mid K}\left(a\mathcal O_x\right)\subset\mathcal O_x\right\},$ which, according to \cite{Re}, Section 25, is an invertible $\mathcal O_x$-module because $\mathcal O_x$ is a maximal order in $A.$

In other words, may view $\widetilde{\mathcal O}$ as the canonical bundle of the extension $\coh(\mathcal O)\map X.$
\end{example}

\begin{example}[Projective lines]
Let $R$ be a ring and let $R[T_0,T_1]$ be the polynomial ring in two indeterminates over $R.$ If $R$ is finitely generated as abelian group, then $\Proj\left(R[T_0,T_1]\right)$ is a noncommutative arithmetic surface. We will show that it is a special case of Example 1.

To do so, we consider the projective line $X=\mathbb P_\mathbb Z^1$ over the integers. Then $\mathcal O=R\otimes_\mathbb Z\mathcal O_X$ is a coherent sheaf of $\mathcal O_X$-algebras and the homogeneous coordinate ring $B=\bigoplus_{n\in\mathbb N}H^0(X,\mathcal O[n])$ is isomorphic to $R[T_0,T_1].$ Indeed, since $X$ is noetherian, cohomology commutes with arbitrary direct sums, thus
$$B\cong H^0(X,\bigoplus_{n\in\mathbb N}\mathcal O[n])\cong H^0(X,R\otimes_\mathbb Z\bigoplus_{n\in\mathbb N}\mathcal O_X[n])\cong R\otimes_\mathbb Z\Gamma_+(\mathcal O_X),$$
where $\Gamma_+(\mathcal O_X)=\bigoplus_{n\in\mathbb N}H^0(X,\mathcal O_X[n])\cong\mathbb Z[T_0,T_1].$ This shows $B\cong R\otimes_\mathbb Z\mathbb Z[T_0,T_1]\cong R[T_0,T_1].$ So it follows from Example 1 that the categories $\coh(\mathcal O)$ and $\qgr\left(R[T_0,T_1]\right)$ are equivalent and that $\Proj\left(R[T_0,T_1]\right)$ is a noncommutative arithmetic surface.

In the polynomial ring $R[T_0,T_1]$ the indeterminates $T_0$ and $T_1$ lie in the center. Of course this is quite a strong restriction in a noncommutative setting, but fortunately it can be omitted by considering twisted polynomial rings. They are obtained as follows. Let $\sigma$ be a graded automorphism of the polynomial ring $R[T_0,T_1].$ A new multiplication on the underlying graded $R$-module $R[T_0,T_1]$ is defined by
$$p\ast q=p\sigma^d(q),$$
where $p$ and $q$ are homogeneous polynomials and $\deg(p)=d.$ The graded ring thus obtained is called twisted polynomial ring and is denoted by $R[T_0,T_1]^\sigma.$ Zhang \cite{Z}  proved that the categories $\Gr R[T_0,T_1]$ and $\Gr R[T_0,T_1]^\sigma$ are equivalent whence $\Proj\left(R[T_0,T_1]\right)\cong\Proj\left(R[T_0,T_1]^\sigma\right).$ So, twisted polynomial rings in two indeterminates over a noncommutative ring $R$ all define the same noncommutative arithmetic surface.
\end{example}

\begin{example}[Plane projective curves]
Let $R$ be a ring and let $p\in R[T_0,T_1,T_2]$ be a homogeneous normal polynomial of positive degree. Recall that an element of a ring is called normal if the principal ideal generated by this element is two-sided. This allows to form the factor ring $A=R[T_0,T_1,T_2]/(p).$ If $R$ is finitely generated as abelian group and $(p')=(p)\cap\mathbb Z[T_0,T_1,T_2]$ is a prime ideal of $\mathbb Z[T_0,T_1,T_2],$ then $\Proj A$ is a noncommutative arithmetic curve.

In order to prove this, we set $S=\mathbb Z[T_0,T_1,T_2]/(p').$ Then $X=\Proj S$ is an arithmetic surface. Moreover we have a natural homomorphism $\phi:S\map A$ of noetherian $\mathbb N$-graded $\mathbb Z$-algebras, and $A$ is finitely generated as $S$-module. Hence the sheaf $\widetilde A$ associated to $A$ is a coherent sheaf of $\mathcal O_X$-algebras, and according to Example 1, $\coh(\widetilde A)$ is a noncommutative arithmetic surface. If $B$ denotes the homogeneous coordinate ring of $\coh(\widetilde A),$ the categories $\qgr B$ and $\coh(\widetilde A)$ are equivalent. But the category $\qgr B$ is also equivalent to the category $\qgr A,$ which shows that $\Proj A$ is a noncommutative arithmetic surface.

As in Example 3, we may consider twisted algebras to omit the restriction that the indeterminates lie in the center.
\end{example}

\subsection{Arithmetic vector bundles}
Next, we are concerned with arithmetic vector bundles on noncommutative arithmetic surfaces. Recall that a hermitian vector bundle on an arithmetic variety $X$ is a pair $\ol E=(E,h)$ consisting of a locally free sheaf $E$ on $X$ and a \emph{smooth} hermitian metric on the holomorphic vector bundle $E_\mathbb C$ induced by $E$ on the associated complex algebraic variety $X_\mathbb C.$

Unfortunately we are not able to adapt this definition of a hermitian vector bundle to our noncommutative setting. The main problem is that we do not know what should be the differential structure on a noncommutative complex algebraic variety. This problem was also mentioned by other authors, for example Polishchuk \cite{P} writes: ``However, it is rather disappointing that at present there is almost no connection between noncommutative algebraic varieties over $\mathbb C$ and noncommutative topological spaces, which according to Connes are described by $C^*$-algebras.''

The lack of smooth hermitian metrics on vector bundles on noncommutative algebraic varieties forces us to substitute the infinite part of a hermitian vector bundle on an arithmetic variety. We replace the hermitian metric by an automorphism of the induced real bundle. In this manner, we obtain objects $\widetilde{E}=(E,\beta)$ consisting of a locally free sheaf $E$ on $X$ and an automorphism $\beta$ of the real vector bundle $E_\mathbb R$ induced by $E$ on the associated real algebraic variety $X_\mathbb R.$ Since the automorphism group of a finite dimensional real vector space $V$ and the group of nondegenerate $\mathbb R$-bilinear forms $V\times V\map\mathbb R$ are isomorphic, the pairs $\widetilde E$ and hermitian vector bundles are closely related. For example, if $L$ is an invertible $\mathbb Z$-module then $L_\mathbb R$ is a one-dimensional real vector space. Therefore every automorphism $\alpha$ of $L_\mathbb R$ is given by multiplication by some nonzero $r\in\mathbb R.$ Let $x$ be a basis of $L$ over $\mathbb Z$ and let $h_\alpha$ denote the scalar product on $L_\mathbb R$ which is defined by the equation $h_\alpha(x,x)=r^2.$ Note that $h_\alpha$ does not depend on the choice of the basis $x$ of $L$ because the only other basis of $L$ is $-x.$ This allows to associate in a well-defined manner the arithmetic line bundle $\ol L=(L,h_\alpha)$ to the pair $\widetilde{L}=(L,\alpha),$ and we may note the following

\begin{remark}
Let $\widetilde L=(L,\alpha)$ and $\widetilde M=(M,\beta)$ be two pairs consisting of an invertible $\mathbb Z$-module and an automorphism of the induced real vector space. The associated arithmetic line bundles $\ol L=(L,h_\alpha)$ and $\ol M=(M,h_\beta)$ are isomorphic if and only if $\left|\det\alpha\right|=\left|\det\beta\right|.$ Moreover, $\adeg\ol L=-\log\left|\det\alpha\right|.$
\end{remark}

So at least for arithmetic curves, we obtain essentially the same theory whether the infinite part is given by a scalar product or by an automorphism. Let us now consider hermitian vector bundles on arithmetic surfaces. If $X$ is an arithmetic surface, the associated complex variety $X_\mathbb C$ is a Riemann surface. It is possible to endow every line bundle $L$ with a distinguished hermitian metric $g.$ This metric is defined via the Green's function for the Riemann surface $X_\mathbb C$ and is therefore called the Green's metric. Already Arakelov \cite{A} observed that the Green's metric is admissible and that every other admissible metric $h$ on $L$ is a scalar multiple of the Green's metric, i.e. there is some positive real number $\alpha$ such that $h=\alpha g.$

On the other hand for every line bundle $L$ on the arithmetic surface $X,$ we have
$$\Hom_{X_\mathbb R}\left(L_\mathbb R,L_\mathbb R\right)\cong\Hom_{X_\mathbb R}\left(\mathcal O_{X_\mathbb R},\mathcal O_{X_\mathbb R}\right)\cong H^0\left(X_\mathbb R,\mathcal O_{X_\mathbb R}\right)\cong\mathbb R.$$
This shows that the set of admissible hermitian line bundles on $X$ embeds into the set of pairs $(L,\alpha)$ consisting of a line bundle and an automorphism of the real line bundle $L_\mathbb R.$

All these considerations motivate

\begin{definition}
An \emph{arithmetic vector bundle} on a noncommutative arithmetic surface $\Proj A$ is a pair $\ol{\mathcal E}=(\mathcal E,\beta)$ consisting of a noetherian object $\mathcal E$ of $\QGr A$ and an automorphism $\beta:\mathcal E_\mathbb R\map\mathcal E_\mathbb R.$ If $\mathcal E$ is an invertible object, then $\ol{\mathcal E}$ is called an \emph{arithmetic line bundle} on $\Proj A.$
\end{definition}

\section{Arithmetic intersection and Riemann-Roch theorem}
To compute the intersection of two arithmetic line bundles on a noncommutative arithmetic surface, we follow Faltings original approach involving the determinant of the cohomology, see \cite{F2}. Recall that the determinant of the cohomology of a vector bundle $E$ on an arithmetic surface $X$ is defined to be the invertible $\mathbb Z$-module
\begin{equation}\label{doc}
\lambda(E)=\det H^0(X,E)\otimes_\mathbb Z\left(\det H^1(X,E)\right)^{-1}.
\end{equation}
If the finitely generated abelian group $H^i(X,E)$ has nontrivial torsion part with cardinality $n>0$ then $\det H^i(E,X)=\frac{1}{n}\det H^i_f,$ where $H^i_f$ is the free part of  $H^i(X,E).$ Given two line bundles $L$ and $M$ on $X,$ one puts
\begin{equation}\label{inter}
\langle L,M\rangle=\lambda(L\otimes M)\otimes_\mathbb Z\lambda(L)^{-1}\otimes_\mathbb Z\lambda(M)^{-1}\otimes_\mathbb Z\lambda(\mathcal O_X).
\end{equation}
This is again an invertible $\mathbb Z$-module and its norm equals the intersection number of $L$ and $M.$

Now, let $\ol L$ and $\ol M$ be two hermitian line bundles on an arithmetic surface $X$ and let $\ol N=(N,h)$ be a hermitian line bundle on a Riemann surface $Y.$ Quillen \cite{Q} and Faltings \cite{F2} developed two different methods to endow the determinant of the cohomology $\lambda(N)$ with a hermitian metric which is naturally induced by the given metric $h.$ The two metrics differ by a constant which only depends on the Riemann surface $Y$ but is independent of the particular hermitian line bundle $\ol N,$ cf. \cite{So}, Section 4.4. However, both approaches have been used to equip the complex vector space $\langle L,M\rangle_\mathbb C$ with a suitable hermitian metric and thus obtaining a hermitian line bundle $\langle\ol L,\ol M\rangle$ on $\Spec\mathbb Z$ which, in view of (\ref{inter}), is a natural generalization of the intersection of $L$ and $M.$

In this vein, we are now going to define arithmetic intersection on noncommutative arithmetic surfaces. Let $\Proj A$ be a noncommutative arithmetic surface and fix some automorphism $\alpha$ of the dualizing object $\omega_\mathbb R.$ Let $\ol{\mathcal L}=(\mathcal L,\beta)$ be an arithmetic line bundle and $\ol{\mathcal E}=(\mathcal E,\gamma)$ be an arithmetic vector bundle on $\Proj A.$ Recall that $\beta$ induces homomorphisms
$$\beta_n^*:\Ext_{\qgr A_\mathbb R}^n(\mathcal L_\mathbb R,\mathcal E_\mathbb R)\map\Ext_{\qgr A_\mathbb R}^n(\mathcal L_\mathbb R,\mathcal E_\mathbb R),\quad n\ge 0.$$
Likewise $\gamma$ induces homomorphisms
$$(\gamma_*)_n:\Ext_{\qgr A_\mathbb R}^n(\mathcal L_\mathbb R,\mathcal E_\mathbb R)\map\Ext_{\qgr A_\mathbb R}^n(\mathcal L_\mathbb R,\mathcal E_\mathbb R),\quad n\ge 0.$$
We set $\beta^*=\beta_0^*$ and $\gamma_*=(\gamma_*)_0.$ To avoid ambiguities, we sometimes write $\Ext_{\qgr A_\mathbb R}^n(\beta,\mathcal E_\mathbb R)$ instead of $\beta_n^*$ and similarly for $(\gamma_*)_n.$

Inspired by the above considerations in the commutative case, we define two arithmetic line bundles on $\Spec\mathbb Z,$ namely
$$\det\Hom\left(\ol{\mathcal L},\ol{\mathcal E}\right)=\left(\det\Hom_{\qgr A}(\mathcal L,\mathcal E),\det\left((\beta^{-1})^*\circ\gamma_*\right)\right)$$
and
$$\det\Ext\left(\ol{\mathcal L},\ol{\mathcal E}\right)=\left(\det\Ext_{\qgr A}^1(\mathcal L,\mathcal E),\det\left((\beta^{-1})_1^*\circ(\gamma_*)_1\right)\det(\alpha_*)^{-1}\right),$$
where $\alpha_*=\Hom_{\qgr A_\mathbb R}\left(t_\mathbb R^{-1}(\mathcal E_\mathbb R),\alpha\right)$ and $t_\mathbb R$ is an automorphism of $\QGr A_\mathbb R$ associated to the invertible object $\mathcal L_\mathbb R.$ By definition of a noncommutative arithmetic surface, all the Ext groups occurring above are finitely generated. Moreover the base change lemma, Lemma \ref{basechange}, ensures that this remains true after tensoring with $\mathbb R,$ i.e. the Ext groups $\Ext_{\qgr A_\mathbb R}^n(\mathcal L_\mathbb R,\mathcal E_\mathbb R)$ are finite dimensional real vector spaces, therefore the determinants are reasonable and everything makes sense.

The expression $\det(\alpha_*)^{-1}$ occurs for the following reason. If $\ol{\mathcal A}=\left(\mathcal A,\id\right)$ is the trivial arithmetic bundle on $\Proj A,$ then
\begin{equation*}
\begin{split}
\left(\det\Ext(\ol{\mathcal A},\ol{\mathcal A})\right)^{-1}&=\left(\left(\det\Ext_{\qgr A}^1(\mathcal A,\mathcal A)\right)^{-1},\det(\alpha_*)\right)\\
&\cong\left(\det\Hom_{\qgr A}(\mathcal A,\omega),\det(\alpha_*)\right)\\
&=\det\Hom\left(\ol{\mathcal A},\ol{\omega}\right),
\end{split}
\end{equation*}
which is a natural condition that is imposed also in the definition of the Faltings metric, see \cite{F2}. In other words, the expression $\det(\alpha_*)$ is necessary in order that the natural isomorphism between Ext groups given by the dualizing sheaf induces an isomorphism of the corresponding arithmetic line bundles. For more details we refer to Theorem \ref{dual} below.

\begin{definition}
Let $\ol{\mathcal L}$ be an arithmetic line bundle and $\ol{\mathcal E}$ be an arithmetic vector bundle on a noncommutative arithmetic surface $\Proj A.$ We set
$$\lambda(\ol{\mathcal L},\ol{\mathcal E})=\det\Hom\left(\ol{\mathcal L},\ol{\mathcal E}\right)\otimes_\mathbb Z\left(\det\Ext\left(\ol{\mathcal L},\ol{\mathcal E}\right)\right)^{-1}.$$
If $\ol{\mathcal L}=\ol{\mathcal A}$ is the trivial arithmetic bundle, then we write $\lambda(\ol{\mathcal E})$ instead of $\lambda(\ol{\mathcal A},\ol{\mathcal E})$ and call it the \emph{determinant of the cohomology} of $\ol{\mathcal E}.$
\end{definition}

Recall that for every $\mathcal O_X$-module $\mathcal E$ on a commutative scheme $X$ and every $i\ge 0,$ the Ext group $\Ext_X^i(\mathcal O_X,\mathcal E)$ and the cohomology group $H^i(X,\mathcal E)$ are naturally isomorphic. Hence in view of (\ref{doc}), our definition of the determinant of the cohomology is a natural generalization of the definition in the commutative case.

In order to prove that the determinant of the cohomology is compatible with Serre duality, we need the following

\begin{lemma}\label{semisimple}
Let $A$ be a finite dimensional simple algebra over a field $k,$ and let $M$ be a finitely generated $A$-bimodule. Given $a\in A,$ let $\lambda_a$ and $\rho_a$ denote left and right multiplication by $a$ on $M,$ respectively. Then for every $a\in A,$
$$\textstyle{\det_k}(\lambda_a)=\det_k(\rho_a).$$
\end{lemma}

\begin{proof}
Let $K$ denote the center of the simple $k$-algebra $A$ and $A^\circ$ its opposite ring. If $A$ has dimension $r$ over $K,$ then it follows from \cite{Re}, Theorem (7.13), that the enveloping algebra $A^e=A\otimes_KA^\circ$ is isomorphic to the algebra $M_r(K)$ of $r\times r$-matrices over the field $K.$ We may view every $A$-bimodule $N$ as a left $A^e$-module, by means of the formula
$$(a\otimes b^\circ)n=anb,\quad\text{for all }a\in A,b^\circ\in A^\circ,n\in N.$$
Since $A^e\cong M_r(K),$ every minimal left ideal $V$ of $A^e$ is isomorphic to the left $A$-module of $r$-component column vectors with entries in $K,$ and hence has dimension $r$ over $K.$ Moreover, every finitely generated left $A^e$-module is isomorphic to a finite direct sum of copies of $V.$ In particular, this applies to the left $A^e$-module $A,$ whence $A\cong V$ for dimensional reasons. This implies that there is a natural number $m$ such that $M\cong A^m$ as $A$-bimodules. Hence for every $a\in A,$ $\det_k(\lambda_a)=m\det_k(f_a)$ and $\det(\rho_a)=m\det_k(g_a),$ where $f_a:A\map A$ is left multiplication and $g_a:A\map A$ is right multiplication by $a.$ Finally, it follows from \cite{Re}, Theorem (9.32), that $\det_k(f_a)=N_{A\mid k}(a)=\det_k(g_a)$ which completes the proof.
\end{proof}

For further use, we note the following straightforward observation:

\begin{remark}\label{det}
Let $k$ be a field, let $F,G:\mathsf C\map\operatorname{mod}$-$k$ be two functors from any category $\mathsf C$ to the category of finite dimensional $k$-vector spaces, let $\eta:F\map G$ be a natural transformation, and let $A$ be an object of $\mathsf C.$ If $\eta_A$ is an isomorphism, then $\det F(f)=\det G(f)$ for every endomorphism $f:A\map A.$
\end{remark}

Now we are ready to prove

\begin{theorem}\label{dual}
Let $\ol{\mathcal L}$ be an arithmetic line bundle on a noncommutative arithmetic surface $\Proj A.$ Suppose that $\Proj A_\mathbb R$ satisfies Serre duality. If $\Hom_{\QGr A_\mathbb R}(\mathcal A_\mathbb R,\mathcal A_\mathbb R)$ is a simple ring, then the determinant of the cohomology is compatible with Serre duality, that is,
$$\lambda(\ol{\mathcal L},\ol{\omega})\cong\lambda(\ol{\mathcal L}).$$
\end{theorem}

\begin{proof}
Recall
$$\lambda(\ol{\mathcal L},\ol{\omega})=\det\Hom(\ol{\mathcal L},\ol{\omega})\otimes\left(\det\Ext(\ol{\mathcal L},\ol{\omega})\right)^{-1}$$
and
$$\lambda(\ol{\mathcal L})=\det\Hom(\ol{\mathcal A},\ol{\mathcal L})\otimes\left(\det\Ext(\ol{\mathcal A},\ol{\mathcal L})\right)^{-1}.$$
If $\ol{\mathcal L}=(\mathcal L,\beta)$ and $\ol{\omega}=(\omega,\alpha)$ then
\begin{equation}\label{dual1}
\det\Hom\left(\ol{\mathcal L},\ol{\omega}\right)=\left(\det\Hom_{\qgr A}(\mathcal L,\omega),\det\left((\beta^{-1})^*\right)\det(\alpha_*)\right)
\end{equation}
and
\begin{equation}\label{dual2}
\left(\det\Ext(\ol{\mathcal A},\ol{\mathcal L})\right)^{-1}=\left(\left(\det\Ext_{\qgr A}^1(\mathcal A,\mathcal L)\right)^{-1},\det(\beta_*)_1^{-1}\det(\alpha_*)\right),
\end{equation}
where $(\beta^{-1})^*=\Hom_{\qgr A_\mathbb R}\left(\beta^{-1},\omega_\mathbb R\right),$ $(\beta_*)_1=\Ext_{\qgr A_\mathbb R}^1\left(\mathcal A_\mathbb R,\beta\right)$ and $\alpha_*=\Hom_{\qgr A_\mathbb R}\left(\mathcal L_\mathbb R,\alpha\right).$ By definition of the dualizing object, the functors $\Ext_{\qgr A_\mathbb R}^1(\mathcal A_\mathbb R,-)^\vee$ and $\Hom_{\qgr A_\mathbb R}(-,\omega_\mathbb R)$ are isomorphic, so according to Remark \ref{det},
\begin{equation}\label{dual3}
\det\left((\beta^{-1})^*\right)=\det\left((\beta^{-1}_*)_1^\vee\right).
\end{equation}
Moreover,
\begin{equation}\label{dual4}
\det\left((\beta^{-1}_*)_1^\vee\right)=\det\left((\beta^{-1}_*)_1\right)=\det(\beta_*)_1^{-1}.
\end{equation}
Combining (\ref{dual1}), (\ref{dual2}), (\ref{dual3}) and (\ref{dual4}) yields
\begin{equation}\label{dual5}
\det\Hom\left(\ol{\mathcal L},\ol{\omega}\right)\cong\left(\det\Ext\left(\ol{\mathcal A},\ol{\mathcal L}\right)\right)^{-1}.
\end{equation}

It remains to prove
\begin{equation}\label{dual6}
\det\Hom\left(\ol{\mathcal A},\ol{\mathcal L}\right)\cong\left(\det\Ext\left(\ol{\mathcal L},\ol{\omega}\right)\right)^{-1}.
\end{equation}
By definition,
\begin{equation}\label{dual7}
\det\Hom\left(\ol{\mathcal A},\ol{\mathcal L}\right)=\left(\det\Hom_{\qgr A}(\mathcal A,\mathcal L),\det(\beta_*)\right)
\end{equation}
and
\begin{equation}\label{dual8}
\left(\det\Ext\left(\ol{\mathcal L},\ol{\omega}\right)\right)^{-1}=\left(\left(\det\Ext_{\qgr A}^1(\mathcal L,\omega)\right)^{-1},\det\left((\beta^{-1})_1^*\circ(\alpha_*)_1\right)^{-1}\det(\alpha_*)\right),
\end{equation}
where $\alpha_*=\Hom_{\qgr A_\mathbb R}\left(t_\mathbb R^{-1}\left(\omega_\mathbb R\right),\alpha\right)$ and $t_\mathbb R$ is an autoequivalence associated to the invertible object $\mathcal L_\mathbb R.$ Since $\Proj A_\mathbb R$ satisfies Serre duality, the functors $\Ext_{\qgr A_\mathbb R}^1(-,\omega_\mathbb R)$ and $\Hom_{\qgr A_\mathbb R}(\mathcal A_\mathbb R,-)^\vee$ are isomorphic, hence Remark \ref{det} yields
\begin{equation}\label{dual9}
\det\left(\beta_1^*\right)=\det\left((\beta_*)^\vee\right).
\end{equation}
But $\det\left((\beta^{-1})_1^*\right)^{-1}=\det\left(\beta_1^*\right)$ and $\det\left((\beta_*)^\vee\right)=\det\left(\beta_*\right),$ so we may combine (\ref{dual7}), (\ref{dual8}) and (\ref{dual9}) to see that
$$\det\Hom(\ol{\mathcal A},\ol{\mathcal L})\cong\left(\left(\det\Ext_{\qgr A}^1(\mathcal L,\omega)\right)^{-1},\det\left((\beta^{-1})_1^*\right)^{-1}\right).$$
Hence in order to prove (\ref{dual6}), it suffices to show
\begin{equation}\label{dual10}
\det\left((\alpha_*)_1\right)^{-1}\det(\alpha_*)=1.
\end{equation}

Since $\mathcal L_\mathbb R$ is invertible, we have natural isomorphisms
$$\Ext_{\qgr A_\mathbb R}^1\left(\mathcal L_\mathbb R,-\right)\cong\Ext_{\qgr A_\mathbb R}^1\left(t_\mathbb R(\mathcal A_\mathbb R),-\right)\cong\Ext_{\qgr A_\mathbb R}^1\left(\mathcal A_\mathbb R,t_\mathbb R^{-1}(-)\right).$$
In combination with the natural isomorphism
$$\Ext_{\qgr A_\mathbb R}^1\left(\mathcal A_\mathbb R,t_\mathbb R^{-1}(-)\right)^\vee\cong\Hom_{\qgr A_\mathbb R}\left(t_\mathbb R^{-1}(-),\omega_\mathbb R\right),$$
which is provided by definition of the dualizing object, we obtain a natural isomorphism
$$\Ext_{\qgr A_\mathbb R}^1\left(\mathcal L_\mathbb R,-\right)^\vee\cong\Hom_{\qgr A_\mathbb R}\left(t_\mathbb R^{-1}(-),\omega_\mathbb R\right).$$
Applying Remark \ref{det} leads to
\begin{equation}\label{dual11}
\det\left((\alpha_*)_1\right)=\det\left((\alpha_*)_1^\vee\right)=\det\left(t_\mathbb R^{-1}(\alpha)^*\right).
\end{equation}

Let $E$ denote the endomorphism ring $\End_{\qgr A_\mathbb R}(\omega_\mathbb R)$ of the dualizing object $\omega_\mathbb R.$ There is an $E$-bimodule structure on the abelian group $G=\Hom_{\qgr A_\mathbb R}\left(t_\mathbb R^{-1}(\omega_\mathbb R),\omega_\mathbb R\right)$ defined by
$$e\varphi f=e\circ\varphi\circ t_\mathbb R^{-1}(f),\quad\text{for all }e,f\in E,\varphi\in G.$$
In terms of this $E$-bimodule structure, applying the homomorphism $\alpha_*= \linebreak \Hom_{\qgr A_\mathbb R}\left(t_\mathbb R^{-1}(\omega_\mathbb R),\alpha\right)$ is simply left multiplication by $\alpha.$ Analogously, applying $t_\mathbb R^{-1}(\alpha)^*$ is right multiplication by $\alpha.$ Therefore, if $E$ is a finite dimensional simple $\mathbb R$-algebra, then it follows from Lemma \ref{semisimple} that $\det(\alpha_*)=\det(t_\mathbb R^{-1}(\alpha)^*),$ which together with (\ref{dual11}) establishes (\ref{dual10}). It thus remains to show that $E$ is simple and finite dimensional over $\mathbb R.$

Since $\Proj A_\mathbb R$ satisfies Serre duality, there exist two natural isomorphisms
$$\theta^i:\Ext_{\qgr A_\mathbb R}^i(-,\omega_\mathbb R)\map\Ext_{\qgr A_\mathbb R}^{1-i}(\mathcal A_\mathbb R,-)^\vee,\quad i=0,1.$$
Therefore
\begin{equation}\label{dual12}
E=\Hom_{\qgr A_\mathbb R}\left(\omega_\mathbb R,\omega_\mathbb R\right)\cong\Ext_{\qgr A_\mathbb R}^1\left(\mathcal A_\mathbb R,\omega_\mathbb R\right)^\vee\cong\Hom_{\qgr A_\mathbb R}\left(\mathcal A_\mathbb R,\mathcal A_\mathbb R\right).
\end{equation}
By assumption, $B=\Hom_{\qgr A_\mathbb R}\left(\mathcal A_\mathbb R,\mathcal A_\mathbb R\right)$ is a simple $\mathbb R$-algebra, whose dimension over $\mathbb R$ is finite because $\Proj A_\mathbb R$ is $H$-finite. The isomorphism in (\ref{dual12}) is an isomorphism of real vector spaces, so it only implies that $E$ has the same finite dimension as $B,$ but it does not ensure that $E\cong B$ as rings.

To verify this last point, we first note that left multiplication $\lambda_b$ on $\Gamma_\mathbb R(\omega_\mathbb R)$ induces the ring homomorphism
$$\lambda:B\map E'=\End_{\Gr A_\mathbb R}\left(\Gamma_\mathbb R(\omega_\mathbb R)\right)^{\circ},\; b\mapsto\lambda_b.$$
Since $E'\neq 0,$ the two-sided ideal $\ker\lambda$ of the simple ring $B$ is proper. This implies that $\lambda$ is injective. On the other hand, according to Lemma \ref{lemma1}(ii), the representing functor $\Gamma_\mathbb R:\QGr A_\mathbb R\map\Gr A_\mathbb R$ is fully faithful, thus the rings $E$ and $E'$ are isomorphic. But $\lambda$ is also a homomorphism of real vector spaces, so it follows from (\ref{dual12}) that $\lambda$ is in fact an isomorphism, which shows that the rings $E$ and $B$ are isomorphic. Hence $E$ is indeed a finite dimensional simple $\mathbb R$-algebra, and the theorem is established.
\end{proof}

We proceed now with the definition of the intersection of arithmetic bundles on noncommutative arithmetic surfaces.
\begin{definition}
Let $\ol{\mathcal L}$ be an arithmetic line bundle and $\ol{\mathcal E}$ be an arithmetic vector bundle on a noncommutative arithmetic surface $\Proj A.$ We set
\begin{equation}\label{inter2}
\langle\ol{\mathcal L},\ol{\mathcal E}\rangle =\lambda(\ol{\mathcal L},\ol{\mathcal E})\otimes_\mathbb Z \lambda(\ol{\mathcal L},\ol{\mathcal A})^{-1} \otimes_\mathbb Z \lambda(\ol{\mathcal E})^{-1} \otimes_\mathbb Z \lambda(\ol{\mathcal A})
\end{equation}
and call it the \emph{intersection} of $\ol{\mathcal L}$ with $\ol{\mathcal E}.$ Moreover,
\begin{equation}\label{thomas}
(\ol{\mathcal L},\ol{\mathcal E})=-\adeg\left(\langle\ol{\mathcal L},\ol{\mathcal E}\rangle\right)
\end{equation}
is called the \emph{intersection number} of $\ol{\mathcal L}$ with $\ol{\mathcal E}.$
\end{definition}
Note that $\langle\ol{\mathcal L},\ol{\mathcal E}\rangle$ is an arithmetic line bundle on $\Spec\mathbb Z$ and $(\ol{\mathcal L},\ol{\mathcal E})$ is the negative of its arithmetic (or Arakelov) degree. The minus sign occurs for the following reason. If $\ol{\mathcal L}$ and $\ol{\mathcal M}$ are two arithmetic line bundles on a commutative arithmetic surface $X,$ then
$$\lambda(\ol{\mathcal L},\ol{\mathcal M})\cong\lambda(\ol{\mathcal O}_X,\ol{\mathcal L}^{-1}\otimes\ol{\mathcal M})=\lambda(\ol{\mathcal L}^{-1}\otimes\ol{\mathcal M}).$$
Putting this into (\ref{inter2}) yields
$$\langle\ol{\mathcal L},\ol{\mathcal M}\rangle\cong\lambda(\ol{\mathcal L}^{-1}\otimes\ol{\mathcal M})\otimes\lambda(\ol{\mathcal L}^{-1})^{-1}\otimes\lambda(\ol{\mathcal M})^{-1}\otimes\lambda(\ol{\mathcal O}_X).$$
Comparing the right hand side with $(\ref{inter}),$ we see that it computes the intersection of $\ol{\mathcal L}^{-1}$ and $\ol{\mathcal M}.$ Since $\adeg\big(\langle\ol{\mathcal L}^{-1},\ol{\mathcal M}\rangle\big)=-\adeg\left(\langle\ol{\mathcal L},\ol{\mathcal M}\rangle\right),$ this explains the minus sign in (\ref{thomas}).

Finally, we are able to prove the following Riemann-Roch theorem:

\begin{theorem}
Let $\ol{\mathcal L}$ be an arithmetic line bundle on a noncommutative arithmetic surface $\Proj A,$ and suppose that $\Proj A_\mathbb R$ satisfies Serre duality. If $\Hom_{\QGr A_\mathbb R}\left(\mathcal A_\mathbb R,\mathcal A_\mathbb R\right)$ is a simple ring, then there is an isomorphism
\begin{equation}\label{RR}
\lambda(\ol{\mathcal L})^{\otimes -2}\otimes\lambda(\ol{\mathcal A})^{\otimes 2}\cong\langle\ol{\mathcal L},\ol{\mathcal L}\rangle\otimes\langle\ol{\mathcal L},\ol{\omega}\rangle^{-1}.
\end{equation}
\end{theorem}

\begin{proof}
Once
\begin{equation}\label{RR1}
\lambda(\ol{\mathcal L},\ol{\mathcal L})\cong\lambda(\ol{\mathcal A})
\end{equation}
is established, the theorem follows immediately. Indeed, if (\ref{RR1}) holds then
\begin{equation}\label{RR2}
\langle\ol{\mathcal L},\ol{\mathcal L}\rangle=\lambda(\ol{\mathcal L},\ol{\mathcal L})\otimes\lambda(\ol{\mathcal L},\ol{\mathcal A})^{-1}\otimes\lambda(\ol{\mathcal L})^{-1}\otimes\lambda(\ol{\mathcal A})\cong\lambda(\ol{\mathcal A})^{\otimes 2}\otimes\lambda(\ol{\mathcal L},\ol{\mathcal A})^{-1}\otimes\lambda(\ol{\mathcal L})^{-1}.
\end{equation}
On the other hand, according to Theorem \ref{dual}, we have $\lambda(\ol{\mathcal L},\ol{\omega})\cong\lambda(\ol{\mathcal L})$ and $\lambda(\ol{\omega})\cong\lambda(\ol{\mathcal A}),$ whence
\begin{equation}\label{RR3}
\langle\ol{\mathcal L},\ol{\omega}\rangle^{-1}=\lambda(\ol{\mathcal L},\ol{\omega})^{-1}\otimes\lambda(\ol{\mathcal L},\ol{\mathcal A})\otimes\lambda(\ol{\omega})\otimes\lambda(\ol{\mathcal A})^{-1}\cong\lambda(\ol{\mathcal L})^{-1}\otimes\lambda(\ol{\mathcal L},\ol{\mathcal A}).
\end{equation}
Combining (\ref{RR2}) and (\ref{RR3}) yields (\ref{RR}).

So let us prove (\ref{RR1}). Since $\mathcal L_\mathbb R$ is invertible, there is an autoequivalence $t_\mathbb R$ of $\QGr A_\mathbb R$ such that $t_\mathbb R(\mathcal A_\mathbb R)\cong\mathcal L_\mathbb R.$ But every equivalence is a fully faithful functor, therefore the rings $E=\End_{\qgr A_\mathbb R}(\mathcal L_\mathbb R)$ and $\End_{\qgr A_\mathbb R}(\mathcal A_\mathbb R)$ are isomorphic, which shows that $E$ is a finite dimensional simple $\mathbb R$-algebra. If $\ol{\mathcal L}=(\mathcal L,\beta)$ then, in terms of the ring structure of $E,$ $\beta_*:E\map E$ is left and $\beta^*$ is right multiplication by $\beta.$ Hence it follows from Lemma \ref{semisimple} that $\det(\beta_*)=\det(\beta^*),$ whence $\det\left((\beta^{-1})^*\beta_*\right)=1,$ which shows
\begin{equation}\label{RR4}
\det\Hom(\ol{\mathcal L},\ol{\mathcal L})\cong\det\Hom(\ol{\mathcal A},\ol{\mathcal A}).
\end{equation}
The same argument also applies to the bundle $\Ext(\ol{\mathcal L},\ol{\mathcal L}).$ More precisely, the abelian group $\Ext_{\qgr A_\mathbb R}^1\left(\mathcal L_\mathbb R,\mathcal L_\mathbb R\right)$ has a natural $E$-bimodule structure for which $(\beta_*)_1$ is left and $\beta_1^*$ is right multiplication by $\beta.$ And since $E$ is a finite dimensional simple $\mathbb R$-algebra, it follows from Lemma \ref{semisimple} that $\det\left((\beta_*)_1\right)=\det\left((\beta)_1^*\right),$ whence
\begin{equation}\label{RR5}
\det\left((\beta^{-1})_1^*(\beta_*)_1\right)=1.
\end{equation}
On the other hand there is a natural isomorphism
$$\Hom_{\qgr A_\mathbb R}\left(t_\mathbb R^{-1}\left(\mathcal L_\mathbb R\right),-\right)\cong\Hom_{\qgr A_\mathbb R}\left(\mathcal A_\mathbb R,-\right),$$
and applying Remark \ref{det} to the automorphism $\alpha:\omega_\mathbb R\map\omega_\mathbb R$ yields
\begin{equation}\label{RR6}
\det\Hom_{\qgr A_\mathbb R}\left(t_\mathbb R^{-1}\left(\mathcal L_\mathbb R\right),\alpha\right)=\det\Hom_{\qgr A_\mathbb R}\left(\mathcal A_\mathbb R,\alpha\right).
\end{equation}
Combining (\ref{RR5}) and (\ref{RR6}) shows
$$\det\Ext\left(\ol{\mathcal L},\ol{\mathcal L}\right)\cong\det\Ext\left(\ol{\mathcal A},\ol{\mathcal A}\right),$$
which together with (\ref{RR4}) establishes (\ref{RR1}) and thus completes the proof.
\end{proof}

To get a Riemann-Roch formula, we still have to introduce the Euler characteristic of an arithmetic vector bundle on a noncommutative arithmetic surface. Let $M$ be a finitely generated $\mathbb Z$-module and suppose that $M$ has a volume. Recall the Euler characteristic
$$\chi(M)=-\log\operatorname{vol}(M_\mathbb R/M)+\log|M_{tor}|.$$
Hence, if $\ol{E}$ is a hermitian vector bundle on $\Spec\mathbb Z,$ then $\chi(\ol{E})=\adeg\ol{E}.$ Now, in analogy with \cite{La}, page 112, we define the \emph{Euler characteristic} of an arithmetic vector bundle $\ol{\mathcal E}$ on a noncommutative arithmetic surface $\Proj A$ to be
$$\chi(\ol{\mathcal E})=\chi\left(\Hom(\ol{\mathcal A},\ol{\mathcal E})\right)-\chi\left(\Ext(\ol{\mathcal A},\ol{\mathcal E})\right)=\adeg\lambda(\ol{\mathcal E}).$$

Using this notion, taking degrees in (\ref{RR}) yields the Riemann-Roch formula
$$\chi(\ol{\mathcal L})=\frac{1}{2}\left((\ol{\mathcal L},\ol{\mathcal L})-(\ol{\mathcal L},\ol{\omega})\right)+\chi(\ol{\mathcal A}).$$
Note that this formula looks exactly like the Riemann-Roch formula for hermitian line bundles on commutative arithmetic surfaces, cf. \cite{Mo}, Theorem 6.13.

\bigskip
\bigskip

\begin{tabbing}
Departement Mathematik\\
ETH Z\"urich\\
R\"amistrasse 101\\
8092 Z\"urich\\
Switzerland\\
Phone: +41 44 632 4442\\
Fax:  +41 44 632 1570\\
Email: borek@math.ethz.ch\\
\end{tabbing}

\end{document}